\documentclass[11pt]{amsart}
\usepackage[all]{xy}
\usepackage{pstricks}
\usepackage{color}

\usepackage{tikz}

\theoremstyle{plain}
\newtheorem{thm}{Theorem}[section]
\newtheorem{theorem}[thm]{Theorem}

\newtheorem{lemma}[thm]{Lemma}
\newtheorem{corollary}[thm]{Corollary}

\theoremstyle{definition}
\newtheorem{remark}[thm]{Remark}

\newtheorem{definition}[thm]{Definition}

\newtheorem{question}[thm]{Question}

\numberwithin{equation}{section}

\newcommand{\sI}{{\mathcal I}}

\newcommand{\C}{{\mathbb C}}

\newcommand{\BP}{{\mathbb P}}
\newcommand{\Q}{{\mathbb Q}}
\newcommand{\R}{{\mathbb R}}

\newcommand{\Z}{{\mathbb Z}}
\newcommand{\PP}{\ensuremath{\mathbb{P}}}

\newcommand{\hol}{\ensuremath{\mathcal{O}}}

\title [Coble's question and inertia groups on surfaces]{Coble's question 
and complex dynamics of inertia groups on surfaces}

\author{Keiji Oguiso} 
\address{Mathematical Sciences, the University of Tokyo, Meguro Komaba 3-8-1, Tokyo, Japan and Korea Institute for Advanced Study, Hoegiro 87, Seoul,
133-722, Korea}
\email{oguiso@ms.u-tokyo.ac.jp}

\author{Xun Yu}
\address{Center for Applied Mathematics, Tianjin University, 92 Weijin Road, Nankai District,
Tianjin 300072, P. R. China.
}
\email{xunyu@tju.edu.cn}

\thanks{The first named author is supported by JSPS Grant-in-Aid (S) 15H05738, JSPS Grant-in-Aid (B) 15H03611, and by KIAS Scholar Program. The second named author is supported by NSFC (No. 11701413).}

\begin{document}

\maketitle

\begin{abstract} We study the inertia groups of some smooth rational curves on 2-elementary K3 surfaces and singular K3 surfaces from the view of topological entropy, with an application to a long standing open question of Coble on the inertia group of a generic Coble surface. 
\end{abstract}

\section{Introduction}\label{sect0}

In this introduction, we assume that the base field is the complex number field $\C$. For an algebraic subset $W$ of a variety $V$, we define 
$${\rm Dec}\, (W) = {\rm Dec}\,(V, W) := \{f \in {\rm Aut}\, (V)\,|\, f(W) = W\}\,\, ,$$ 
$${\rm Ine}\, (W) = {\rm Ine}\,(V, W) := \{f \in {\rm Dec}\, (V, W)\,|\, f|_{W} = id_{W}\}\,\, .$$
We call the groups ${\rm Dec}\, (V, W)$ and ${\rm Ine}\, (V, W)$ the decomposition group of $W \subset V$ and the inertia group of $W \subset V$. These two groups for curves on surfaces will play essential roles in this paper.

Let $C \subset \BP^2$ be a generic nodal sextic plane curve with ten nodes. The classical Coble surface $Y = Y_{C}$ is the blowings up  
$$\pi: Y = Y_{C} \to \BP^2$$ 
at the ten nodes of $C$. Let $B \subset Y$ be the proper transform of $C$. Then $B \simeq \BP^1$ and $B$ is the unique element of $|-2K_Y|$ by the adjunction formula and $(-2K_Y)^2 = -4$. Therefore $f(B) = B$ for any $f \in {\rm Aut}\, (Y)$, i.e., ${\rm Aut}\, (Y) = {\rm Dec}\, (Y, B)$ and we obtain a natural group homomorphism
$$\rho_Y : {\rm Aut}\, (Y) \to {\rm Aut}\, (B) = {\rm PGL}\, (2, \C)\,\, ;\,\, 
f \mapsto f|_B\,\, .$$
This homomorphism is first considered by Coble and has attracted many authors since then (See \cite{AD18} for long history and see also \cite{DZ01}, \cite{DK13} for other interesting aspects of Coble surfaces). In particular, the following natural question (\cite[Page 245]{Co19}, see also \cite[Section 4]{AD18}) remains unsolved since Coble asked around 1919:

\begin{question}\label{ques01} Is $\rho_Y$ injective? If otherwise, what can one say about ${\rm Ker}\, \rho_Y$, i.e., ${\rm Ine}\, (B)$?
\end{question}

The primary aim of this paper is to give the following answer of alternative type to this question:

\begin{theorem}\label{thm01} Either $\rho_Y$ is injective or ${\rm Ine}\, (B)$ contains a non-commutative free subgroup isomorphic to $\Z * \Z$ 
and an element of positive entropy. 
\end{theorem}

Our approach is in some sense indirect. Indeed, we deduce Theorem \ref{thm01} from our study of 2-elementary K3 surfaces (Theorem \ref{thm02}). The notion of 2-elementary K3 surface is introduced by Nikulin \cite[Section 4]{Ni81}).

We call a K3 surface $X$ 2-elementary if ${\rm NS}\, (X)^*/{\rm NS}\, (X) \simeq (\Z/2\Z)^{\oplus a}$ for some positive integer $a = a(X)$. Then $a(X) \le \rho(X) := {\rm rank}\, {\rm NS}\, (X)$ and $X$ has the involution $\theta$ such that 
$$\theta^{*}|_{{\rm NS}\, (X)} = id_{{\rm NS}\, (X)}\,\, ,\,\, \theta^{*}\omega_X = -\omega_X\,\, .$$
Here $\omega_X$ is a nowhere vanishing holomorphic 2-form on $X$. Note that $\theta$ is in the center of ${\rm Aut}\, (X)$ (See \cite[Section 4]{Ni81} and Section \ref{sect5} for basic properties of 2-elementary K3 surfaces). 

Our actual main theorem is the following:

\begin{theorem}\label{thm02} Let $X$ be a 2-elementary K3 surface such that $\rho(X) + a(X) = 22$, hence $\rho(X) \ge 11$. Then: 
\begin{enumerate}
\item If $\rho(X) \ge 12$, then there exists a smooth rational curve $C \subset X$ such that ${\rm Ine}\,(C)$ contains a non-commutative free subgroup isomorphic to $\Z * \Z$ and an element of positive entropy.
\item If $\rho(X) = 11$, then $C := X^{\theta}$ is a smooth rational curve (\cite[Section 4]{Ni81}) and ${\rm Ine}\,(C)$ contains a non-commutative free subgroup isomorphic to $\Z * \Z$ and an element of positive entropy unless ${\rm Ine}\, (C) = \{id_X, \theta\}$.
\end{enumerate}
\end{theorem}

Let us return back to our classical Coble surface $Y$. Consider the finite double cover $p : \tilde{Y} \to Y$ branched along $B \in |-2K_Y|$. Then $\tilde{Y}$ is a 2-elementary K3 surface of Picard number $\rho(\tilde{Y}) = 11$ and $a(\tilde{Y}) = 11$ with the covering involution of $p$ as $\theta$. Moreover, we have
$${\rm Aut}\, (\tilde{Y}) = {\rm Dec}\, (\tilde{Y}, \tilde{B})\,\, ,\,\, {\rm Dec}\,(Y, B) = {\rm Aut}\,(Y) = {\rm Aut}\, (\tilde{Y})/\langle \theta \rangle\,\, .$$
Here $\tilde{B} \simeq B \simeq \BP^1$ is the ramification divisor of $p$, i.e., $\tilde{B} = \tilde{Y}^{\theta}$. Theorem \ref{thm01} is then an obvious consequence of Theorem \ref{thm02} (2). 

Our theorem \ref{thm02} (1) is a generalization of 
\cite[Theorem 1.2]{Og18} and also gives a complete affirmative answer to the question in \cite[Remark 4.4]{Og18} when $\rho(X) \ge 12$.

It happens that ${\rm Aut}\, (X)$ has no element of positive entropy for some 2-elementary K3 surface $X$. For instance, any generic K3 surface $X$ of degree $2$ is a 2-elementary K3 surface with ${\rm Aut}\, (X) \simeq \Z/2\Z$ (so that no automorphism of positive entropy). The condition $\rho(X) + a(X) = 22$ is the condition that guarantees that $X^{\theta} \not= \emptyset$ and ${\rm Aut}\, (X)$ has an element of positive entropy (see Lemma \ref{lem:22}). 

It is also interesting to consider a similar question for inertia groups of singular K3 surfaces, i.e., complex K3 surfaces of maximum Picard number $20$ (\cite{SI77}). Recall that the automorphism group of a singular K3 surface always contains an element of positive entropy (\cite[Theorem 1.6 (1)]{Og07}). 

In this direction, we have the following answer, which is also a generalization of \cite[Theorem 1.3]{Og18} (See also \cite[Remark 5.5]{Og18}):

\begin{theorem}\label{thm03} Every singular K3 surface $X$ has a smooth rational curve $C$ such that ${\rm Ine}\,(C)$ contains a non-commutative free subgroup isomorphic to $\Z * \Z$ and an element of positive entropy.
\end{theorem}

Theorems \ref{thm02} and \ref{thm03} are proved in a fairly uniform way in Sections \ref{sect4}, \ref{sect5} as an application of general criteria on the existence of positive entropy element in an inertia group of a K3 surface (Theorem \ref{thm31} and Corollaries \ref{cor31}, \ref{cor32}) in Section \ref{sect2}. Proof of these criteria are based on the Tits' alternative type result for K3 surface automorphism groups (\cite{Og06}, \cite{Og07}, see also Theorem \ref{thm30}). We believe that these criteria will be also applicable for dynamical studies of other K3 surfaces. In Section \ref{sect1}, we prove some constraint of the existence of an element of positive entropy in an inertia group ${\rm Ine}\, (S, C)$ of a curve $C$ on a smooth projective surface $S$ (Theorem \ref{thm21}). This will explain one of the reasons why we may seek for the inertia group of a smooth rational curve on a K3 surface.  
\vskip 4pt
\noindent
{\bf Acknowledgements.} This work has been done during the authors' stay at KIAS  on March 2019. We would like to thank for KIAS and Professors Jun-Muk Hwang and JongHae Keum for invitation, discussions and warm hospitalities. We would like to thank Professor Igor Dolgachev for sending us a very interesting preprint \cite{AD18} by which our work is much inspired. 

\section{Existence of elements of positive entropy in inertia groups of smooth projective surfaces}\label{sect1}

We call an irreducible reduced projective curve $C$ simply a curve. We do not assume that $C$ is smooth. 
We denote by $p_a(C) = h^1(C, \hol_C) = h^0(C, \omega_C)$ the arithemetic genus of $C$ and by $g(C)$ the geometric genus of the normalization $\tilde{C}$ of $C$, i.e., $g(C) = h^1(\tilde{C}, \hol_{\tilde C}) = h^0(\tilde{C}, \omega_{\tilde{C}})$.

Let $S$ be a smooth projective surface defined over an algebraically closed field $k$ and let $C \subset S$ be a curve. 

We denote by ${\rm Bir}\, (S)$ the birational automorphism group of $S$, i.e., the group of birational selfmaps of $S$. We define the subgroups ${\rm BirDec}\, (C)$ and ${\rm BirIne}\, (C)$ of ${\rm Bir}\, (S)$ called the (birational) decomposition group and the (birational) inertia group of $C \subset S$ by 
$${\rm BirDec}\, (C) := \{f \in {\rm Bir}\, (S)\, |\, f_{*}(C) = C \}\,\, ,$$
$${\rm BirIne}\, (C) := \{f \in {\rm BirDec}\, (S)\, |\, f|_{C} = id_{C}\}\,\, .$$
Here $f_*(C)$ is the proper transform of $C$, i.e., the Zariski closure of $f(C \setminus I(f))$ where $I(f)$ is the indeterminacy locus of $f$. Note that $I(f)$ consists of at most finitely many closed points. So, the condition $f_{*}(C) = C$ and the condition $f|_{C} = id_{C}$ are well-defined conditions for $f \in {\rm Bir}\, (S)$. By definition, we have a natural group homomorphism 
$$\tilde{\rho} : {\rm BirDec}\, (C) \to  {\rm Bir}\, (C)\,\, ;\,\, f \mapsto f|_C$$
and 
$${\rm BirIne}\, (C) = {\rm Ker}\,\tilde{\rho}\,\, .$$
In particular, ${\rm BirIne}\, (C)$ is a normal subgroup of ${\rm BirDec}\, (C)$. 
We denote by ${\rm Aut}\, (S)$ the biregular automorphism group of $S$ 
and define
$${\rm Dec}\, (C) := {\rm BirDec}\, (C) \cap {\rm Aut}\, (S)\,\, ,\,\, {\rm Ine}\, (C) := {\rm BirIne}\, (C) \cap {\rm Aut}\, (S)\,\, .$$
By restricting $\tilde{\rho}$ to ${\rm Dec}\, (C)$, we obtain the group homomorphism 
$$\rho = \tilde{\rho}|_{{\rm Dec}\, (C)} : {\rm Dec}\, (C) \to {\rm Aut}\, (C)\,\, .$$
Then again, by definition,  
$${\rm Ine}\, (C) = {\rm Ker}\,\rho$$
and ${\rm Ine}\, (C)$ is also a normal subgroup of ${\rm Dec}\, (C)$. 

Coble's question (\cite[Page 245]{Co19}, see also \cite[Section 4]{AD18}) is the one asking the complexity of the actions of  
$${\rm Im}\, \tilde{\rho}\,\, , \,\,{\rm BirIne}\,(C) = {\rm Ker}\, \tilde{\rho}\,\, ,\,\, {\rm Im}\, \rho\,\, ,\,\, {\rm Ine}\, (C) = {\rm Ker}\, \rho$$ 
on $C \subset S$. 

Recall that the first dynamical degree $d_1(f)$ of $f \in {\rm Bir}\, (S)$ is a fundamental measure of the complexity of the action of the itarations $f^n$ ($n \in \Z_{\ge 0}$). It is defined by
$$d_1(f) := \lim_{n \to \infty} |(f^n)^*|_{{\rm End}_{\R}\,({\rm NS}\,(S)_{\R})}^{1/n}\,\, .$$
Here ${\rm NS}\,(S)_{\R} = {\rm NS}\,(S) \otimes \R$ and $|*|_{{\rm End}_{\R}\,({\rm NS}\,(S)_{\R})}$ is any norm of the vector space ${\rm End}_{\R}\,({\rm NS}\,(S)_{\R})$ consisiting of the linear selfmaps of ${\rm NS}\, (S)_{\R}$. By the Gromov-Yomdin's theorem, the topological entropy of $f \in {\rm Aut}\, (S)$ (for a smooth complex projective surface $S$) is given by 
$$h_{{\rm top}}(f) = \log d_1(f)\,\, .$$ 
See \cite[Pages 1637--1639]{DS05} for generalities of dynamical degrees and entropy (see also \cite{DF01} for surface case and \cite{ES13} in positive characteristic). Taking this into account, we set
$$h(f) := \log d_1(f)$$ 
also for $f \in {\rm Bir}\, (S)$ and call $h(f)$ the algebraic entropy (or just entropy) of $f$. We are particularly interested in the existence of positive entropy element of ${\rm BirIne}\, (C)$ and ${\rm Ine}\, (C)$. 

The following theorem, which should be known to the experts, shows that the existence of positive entropy element of ${\rm BirIne}\, (C)$ already poses a fairly strong constraint on the pair $C \subset S$: 

\begin{theorem}\label{thm21} Let $S$ be a smooth projective surface defined over an algebraically closed field of characteristic $0$ and let $C \subset S$ be a curve (hence irreducible and reduced by our conventions) on $S$. Assume that there exists $f \in {\rm BirIne}\, (C)$ such that $h(f) > 0$. Then, one of the following (1) or (2) holds:
\begin{enumerate}
\item $S$ is birational to a K3 surface or an Enriques surface and $C$ is a smooth rational curve; 
\item $S$ is a rational surface and $g(C) = 0$ or $1$. In particular, $C$ is either a rational curve or an elliptic curve.   
\end{enumerate}
\end{theorem}

\begin{remark}\label{rem21} 
\begin{enumerate}
\item There exist a smooth rational surface and a smooth elliptic curve $C \subset S$ such that ${\rm Ine}\, (C)$ has an element of positive entropy. This is proved by Blanc \cite[Section 2]{Bl13}. 
\item There exist a smooth rational surface and a rational curve $C \subset S$ such that $p_a(C) \ge 2$ (hence singular) and ${\rm BirIne}\, (C)$ has an element of positive entropy. For instance, Allcock and Dolgachev \cite[Theorem 6.2]{AD18} give some explicit examples. This paper is much inspired by their examples. 
\end{enumerate} 
\end{remark}

\begin{proof} The existence of $f \in {\rm Bir}\, (S)$ with $h(f) >0$ implies that $S$ is birational to either (i) an abelian surface, (ii) a K3 surface, 
(iii) an Enriques surface or (iv) a rational surface. 

In fact, if the Kodaira dimension $\kappa (S)$ is greater than or equal to $0$, then the minimal model $S_{{\rm min}}$ of $S$ is unique up to isomorphisms and $f$ induces a biregular automorphism $f_{{\rm min}} \in {\rm Aut}\,(S_{{\rm min}})$. By the birational invariance of the dynamical degrees due to Dinh-Sibony \cite[Corollaire 7]{DS05} (See also \cite{DF01} for surface case), we have 
$$h_{{\rm top}}(f_{{\rm min}}) = h(f_{{\rm min}}) = h(f) > 0\,\, .$$ 
Therefore, if  $\kappa(S) \ge 0$, then $S_{{\rm min}}$ is a surface in (i)-(iii) by \cite[Proposition 1]{Ca99}. 

If $\kappa(S) = -\infty$, then $S$ is either a rational surface or 
a birationally ruled surface $\pi : S \to C$ over a smooth curve $C$ of $g(C) > 0$. In the second case, the fibration $\pi$ is preserved by ${\rm Bir}\, (S)$, because all rational curves on $S$ are in fibers of $\pi$ by $g(C) >0$. Therefore, in the second case, $h(f) = 0$ for all $f \in {\rm Bir}\, (S)$ by the product formula due to Dinh-Nugyen \cite[Theorem 1.1]{DN11}. 

Thus, $S$ is birational to a surface in (i) -(iv).

First we consider the cases (i), (ii), (iii). 

Let $\pi : S \to T := S_{{\rm min}}$ be the minimal model of $S$. Then $\pi$ is a composition of blowings down of $(-1)$-curves. As remarked above, $f \in {\rm Bir}\, (S)$ descends to $f_T \in {\rm Aut}\, (T)$ equivariantly with respect to $\pi$ 
and $h(f_T) = h(f) > 0$. Set $C_T := \pi(C)$. 

Assume first that $C_T$ is a point. Then $C$ is one of the exceptional curves of $\pi$. Therefore $C \simeq \BP^1$ on $S$. 

We show that $T$ is not an abelian surface in this case. 
Assuming otherwise, we choose $C_T =O$ as the origin of $T$. We may assume without loss of generality that $T$ is a complex abelian surface, hence it is a complex $2$-torus $T = \C^2/\Lambda$ and that the exceptional locus of $\pi : S \to T$ is connected. The automorphism $f_T$ is then a group automorphism, represented by a linear $2 \times 2$-matrix, say $A$, with respect to the natural holomorphic coordinates $(z_1, z_2)$ of the universal covering space $\C^2$. We denote by $\alpha$ and $\beta$ the eigenvalues of $A$ and arrange so that $|\alpha| \ge |\beta|$. Let $\omega_T$ is a nowhere vanishing holomorphic $2$-form on $T$. Then $f_T^*\omega_T = \alpha\beta\omega_T$. Hence $|\alpha \beta| = 1$ by the finiteness of the canonical representation of a smooth complex projective variety (\cite[Theorem 14.10]{Ue75}). The positivity of the entropy says that the spectral radius of $f_T^*|_{H^{1, 1}(T)}$ is strictly greater than $1$. Hence $|\alpha| >1$ and therefore 
$$|\alpha| > 1 > |\beta|\,\, .$$ 
Let $E_1$ be the exceptional curve of the first blow-up $T_1 \to T$ at $O$ in $\pi$. Then the action of $f_{T_1}$ on $T_1$ of $f$ preserves $E_1$ and has a fixed point $P_1$ over which our $C$ lies. By the property of the blow up, 
the action of $f_{T_1}|_{E_1}$ is either one of: 
$$x_1 \mapsto \frac{\alpha}{\beta} \cdot x_1\,\, ,\,\, x_1 \mapsto \frac{\beta}{\alpha} \cdot x_1$$ 
under a suitable affine coordinate $x_1$ of $E_1$ at $P_1$. More precisely, the action $f_{T_1}$ at $P_1 \in T_1$ is biregular and of the form $(x_1, y_1) \mapsto (a_1x_1, b_1y_1)$ which is either 
one of: 
$$(x_1, y_1) \mapsto (\frac{\alpha}{\beta} x_1, \beta y_1)\,\, ,\,\, (x_1, y_1) \mapsto (\frac{\beta}{\alpha} x_1, \alpha y_1)$$ 
under suitable local coordinates $(x, y)$ of $T_1$ at $P_1$ such that $E_1 = (y_1=0)$ at $P_1$. Here we have still either $|a_1| > 1 > |b_1|$ or $|b_1| > 1 > |a_1|$. Let $E_2$ be the exceptional curve of the second blow-up $T_2 \to T_1$ at $P_1$ in $\pi$. Then the action of $f_{T_2}$ on $T_2$ of $f$ preserves $E_2$ and has a fixed point $P_2$ over which our $C$ lies. For the same reason as above, the action $f_{T_2}$ at $P_2 \in T_2$ is biregular and of the form $(x_2, y_2) \mapsto (a_2x_2, b_2y_2)$ which is either 
one of: 
$$(x_2, y_2) \mapsto (\frac{a_1}{b_1} x_2, b_2 y_2)\,\, ,\,\, (x_2, y_2) \mapsto (\frac{b_1}{a_1} x_2, a_1 y_2)$$ 
under suitable local coordinates $(x_2, y_2)$ of $T_2$ at $P_2$ such that $E_2 = (y_2=0)$ at $P_1$. So, the same condition either $|a_2| > 1 > |b_2|$ or $|b_2| > 1 > |a_2|$ still holds. Now one can repeat this process inductively by reaching the stage that $C$ appears as the exceptional curve $E_n$ of the blow up $T_n \to T_{n-1}$ in $\pi$. The induced action of $f_{T_n}$ on $T_n$ of $f$ preserves $E_n$ and $f_{T_n}|_{E_n}$ is then the multiplication by either 
$$\frac{a_{n-1}}{b_{n-1}}\,\, ,\,\, \frac{b_{n-1}}{a_{n-1}}\,\, .$$ 
However, by induction, we see that either $|a_{n-1}| > 1 > |b_{n-1}|$ or $|b_{n-1}| > 1 > |a_{n-1}|$ holds, so that $f_{T_n}|_{E_n} \not= id_{E_n}$. However, this contradicts to the fact that $f|_{C} = id_{C}$. Indeed, the actions of $f_{T_n}$ on the generic scheme point of $E_n$ and the action of $f$ on the generic scheme point of $C$ have to be the same. So, $T$ is not an abelian surface when $C_T$ is a point. 

Assume that $C_T$ is a curve. Then $f_T^{*}(C_T) = C_T$ in ${\rm NS}\, (T)$. 
As $h(f_T) > 0$, it follows that $C_T^2 < 0$. Here we used the fact that $f_T^*$ is of finite order if $C_T^2 > 0$ and the eigenvalues of $f_T^*$ are all on the unit circle $S^1$ if $C_T^2 = 0$ and $C_T$ is non-zero effective 
(See eg. \cite[Lemma 2.8]{Og07}). 

Since there is no curve with negative self-intersection on an abelian surface, $T$ is not an abelian surface, either. On the other hand, on a K3 surface and an Enriques surface, any curve with negative self-intersection is exactly a $(-2)$-curve and it is isomorphic to $\BP^1$. Hence $C_T \simeq \BP^1$ on $T$ and hence $C \simeq \BP^1$ when $S$ is birational to a K3 surface or an Enriques surface.  

We now consider the case where $S$ is a rational surface. Since the statement and the conclusion are birationally invariant ones, we may assume without loss of generality that $S = \BP^2$. If $g(C) \ge 2$ and $f \in {\rm BirIne}(C)$, then, by Castelnouvo's theorem (See \cite[Th\'eor\`eme 1.1]{BPV08} for the statement and a modern proof), $f$ is either of finite order or a de Jongqui\`eres transformation, which is a birational self map of $\BP^2$ preserving a pencil of rational curves. However, then $h(f) = 0$, respectively by the definition of $h(f)$ and by the product formula as above, a contradiction to $h(f) > 0$. Hence $g(C) \le 1$ as claimed. 

This completes the proof of Theorem \ref{thm21}.
\end{proof}

\section{A criterion of the existence of elements of positive entropy in inertia groups of K3 surfaces}\label{sect2}

Our main results of this section are Theorem \ref{thm31} and Corollaries \ref{cor31} and \ref{cor32}. 

A group $G$ is called almost abelian, if $G$ is isomorphic to an abelian group up to finite kernel and finite cokernel. More precisely, a group $G$ is called almost abelian if there exist a normal subgroup $G^{(0)}$ of $G$ such that $[G:G^{(0)}] < \infty$ and a finite normal subgroup $K < G^{(0)}$ such that the quotient group $G^{(0)}/K$ is an abelian group. We call $G$ almost abelian of 
rank $r$ 
if in addition that we can make $G^{(0)}/K \simeq \Z^{\oplus r}$. The rank $r$ is well-defined (See eg. \cite[Section 8]{Og08}).

\begin{definition}\label{def}
Let $S$ be a smooth projective surface and let $\phi : S \to B$ be a surjective morphism to a smooth projective curve $B$ with connected fibers. We call $\phi$ a genus one fibration if general fibers are of arithmetic genus one. We call $\phi$ an elliptic fibration if general fibers are smooth elliptic curve and $\phi$ admits a global section.  
\end{definition}

Recall from \cite[Theorem 1.1]{Og06} and \cite[Theorem 1.3]{Og07} 
the following:

\begin{theorem}\label{thm30} 
Let $S$ be a projective K3 surface defined over an algebraically closed field $k$ of characteristic $p \not= 2, 3$. 
Let $G$ be any subgroup of ${\rm Aut}\, (S)$. Then:

\begin{enumerate}
\item Either (i) $G$ is an almost abelian group, necessarily of finite rank, or (ii) $G$ contains a subgroup isomorphic to the non-commutative free group $\Z * \Z$, and the two cases (i) and (ii) are exclusive each other. 

\item $G$ has an element of positive entropy in the case (ii). 

\item In particular, if $G$ has no element of positive entropy, then $G$ is almost abelian, i.e., $G$ belongs to the case (i) (by (1) and (2)). Moreover, if $|G| = \infty$, then $G$ has no element of positive entropy if and only if $G$ preserves a genus one fibration $S \to \BP^1$.
\end{enumerate} 
\end{theorem}

\begin{remark}\label{rem30}
Statements of \cite[Theorem 1.1]{Og06} and \cite[Theorem 1.3]{Og07} are formulated over $\C$. However, proofs there are valid without any change over any algebraically closed field $k$ under the assumption that $p \not= 2$, $3$. The assumption $p \not= 2$, $3$ is used only to guarantee that a general fiber of a genus one fibration is a smooth elliptic curve and that the sum of Euler numbers of singular fibers is exactly $24$, the Euler number of a K3 surface, to deduce that a genus one fibration has always at least three singular fibers. 
\end{remark}

In this section, we prove the following:

\begin{theorem}\label{thm31} 
Let $S$ be a projective K3 surface defined over an algebraically closed field $k$ of characteristic $p \not= 2$, $3$ 
and let $C \subset S$ be 
a smooth rational curve. 
Assume that ${\rm Dec}\, (C)$ is not almost abelian and that $|{\rm Ine}\, (C)| = \infty$. Then ${\rm Ine}\,(C)$ contains an element of positive entropy.
\end{theorem}

\begin{proof}
If ${\rm Ine}\, (C)$ has no element of positive entropy, then, since $|{\rm Ine}\, (C) | = \infty$, the group ${\rm Ine}\, (C)$ preserves a genus one fibration $\Phi_{|F|} : S \to \BP^1$ and ${\rm Ine}\, (C)$ is almost abelian of positive finite rank by Theorem \ref{thm30} (3). Here $F$ be a general fiber of 
$\Phi_{|F|}$. 

Let $g \in {\rm Dec}\, (C)$. We have 
$$g^{-1} {\rm Ine}\, (C) g = {\rm Ine}\, (C)\,\, ,$$
because ${\rm Ine}\, (C)$ is a normal subgroup of ${\rm Dec}\, (C)$. 
Thus, ${\rm Ine}\, (C)$ also preserves a genus one fibration $\Phi_{|g^*F|} : S \to \BP^1$. 

Assume that there is $g \in {\rm Dec}\, (C)$ such that $g^*F \not= F$ in ${\rm Pic}\, (S) \simeq {\rm NS}\, (S)$. Then the action of ${\rm Ine}\, (C)$ on ${\rm Pic}\, (S)$ has to preserve the class $g^*F + F$. Then, since
$$(g^*F + F) ^2 = 2(g^*F.F) >0\,\, ,$$ 
the action of ${\rm Ine}\, (C)$ on ${\rm Pic}\, (S)$ is finite (See eg. \cite[Lemma 2.8]{Og07}), hence the group ${\rm Ine}\, (C)$ is a finite group, a contradiction.

Now we may assume that  $g^*F = F$ in ${\rm Pic}\, (S) \simeq {\rm NS}\, (S)$ for all $g \in {\rm Dec}\, (C)$. Then the group ${\rm Dec}\, (C)$ preserves a genus one fibration $\Phi_{|F|} : S \to \BP^1$. However, then ${\rm Dec}\, (C)$ is almost abelian of finite rank by Theorem \ref{thm30} (3), a contradiction to the assumption that ${\rm Dec}\, (C)$ is not almost abelian. 

Hence there is an element $f \in {\rm Ine}\, (C)$ such that $h(f) > 0$, as claimed. 
\end{proof}
We denote ${\rm Dec}\, (S, C, P) := \{g \in {\rm Dec}\, (C)\, |\, g(P) = P\}$ for 
$P \in C \subset S$ for a smooth projective surface $S$, a curve $C \subset S$ and a closed point $P \in C$. Then ${\rm Dec}\,(S, C, P)$ is a subgroup of ${\rm Dec}\, (C)$ and ${\rm Ine}\,(C)$ is a subgroup of ${\rm Dec}\, (S, C, P)$. 

\begin{corollary}\label{cor31} 
Let $S$ be a projective K3 surface defined over 
an algebraically closed field $k$ of characteristic $p \not= 2, 3$ 
and let $C \subset S$ be 
a smooth rational curve. 
Assume that ${\rm Dec}\, (S, C, P)$ is not almost abelian for some point $P \in C$. Then ${\rm Ine}\,(C)$ contains a non-commutative free subgroup isomorphic to $\Z * \Z$ and an element of positive entropy.
\end{corollary}

\begin{proof}
Note that ${\rm Aut}\, (C, P) := \{f \in {\rm Aut}\, (C)\,|\, f(P) = P\}$
is isomorphic to the group of affine linear transformations $f(z) = az + b$ ($a \in k^{\times}$, $b \in k$) of 
the affine line $k$. In particular, ${\rm Aut}\, (C, P)$ fits in with the exact sequence
$$0 \to k \to {\rm Aut}\, (C, P) \to k^{\times} \to 1\,\, .$$
Therefore ${\rm Aut}\, (C, P)$ is solvable, as so are $k$ and 
$k^{\times}$ (indeed, both are abelian groups). 
From the natural representaion $\rho : {\rm Dec}\, (S, C, P) \to {\rm Aut}\, (C, P)$ defined by $f \mapsto f|_C$, we obtain the exact sequence
$$1 \to {\rm Ine}\, (C) \to {\rm Dec}\, (S, C, P) \to {\rm Im}\, \rho \to 1\,\, .$$ Since ${\rm Dec}\, (S, C, P)$ is not almost abelian, ${\rm Dec}\, (S, C, P)$ contains a subgroup $G$ isomorphic to $\Z * \Z$ by Theorem \ref{thm30} (1). The group $\rho(G)$ is solvable, as it is a subgroup of the solvable group ${\rm Aut}\, (C, P)$. Since $G$ is not solvable, it follows then that ${\rm Ker}\, (\rho|_G)$ is not isomorphic to $\{0\}$ nor $\Z$. On the other hand, since $G$ is a free group, the subgroup ${\rm Ker}\, (\rho|_G)$ is also a free group by the Nielsen-Schreier theorem. Thus ${\rm Ker}\, (\rho|_G)$, which is a free group other than $\{0\}$ and $\Z$, has a subgroup isomorphic to $\Z * \Z$. Since ${\rm Ker}\, (\rho|_G) \subset {\rm Ine}\, (C)$, it follows that ${\rm Ine}\, (C)$ also contains a free subgroup isomorphic to $\Z * \Z$. Hence, by Theorem \ref{thm30} (2), ${\rm Ine}\, (C)$ has an element of positive entropy. 
\end{proof}

\begin{remark}\label{rmk:35}
Let $\Phi_{|F|} : S \to \BP^1$ be a genus one fibration on a K3 surface over an algebraically closed field of characteristic $p \not= 2$, $3$. The associated Jacobian fibration $J(\Phi_{|F|}) : J(S) \to \BP^1$ is defined by the relatively minimal model of the compactification of the Jacobian ${\rm Pic}^{0}(S_{\eta})$ of the scheme generic fiber $S_{\eta}$ of $\Phi_{|F|}$. Then $J(\Phi_{|F|})$ is an elliptic fibration and
$${\rm MW}\,(J(\Phi_{|F|})) := S_{\eta}(k(\BP^1))$$ 
forms a finitely generated abelian group called the Mordell-Weil group of 
$J(\Phi_{|F|})$ (See \cite{Shi90} for the basic properties of Mordell-Weil groups). The Mordell-Weil group $ {\rm MW}\,(J(\Phi_{|F|}))$ faithfully acts on $\Phi_{|F|} : S \to \BP^1$ over $\BP^1$ through the translation action of ${\rm Pic}^{0}(S_{\eta})$ on $S_{\eta}$. Note also that ${\rm MW}\, (J(\Phi_{|F|}))$ is a finite index abelian subgroup of the group ${\rm Aut}\,(S)_{|F|}$ in \cite[Section 2.2]{Yu18}. This is because $\Phi_{|F|}$ has at least three singular fibers, so that $|{\rm Im}\,({\rm Aut}\,(S)_{|F|} \to {\rm Aut}\, (\BP^1))| < \infty$, and also $|{\rm Aut}_{{\rm group}}\,(S_{\overline{\eta}})| < \infty$ as $S_{\overline{\eta}}$ is a smooth elliptic curve (cf. Remark \ref{rem30}). 

In what follows, we denote ${\rm MW}\,(J(\Phi_{|F|}))$ simply by ${\rm MW}\, (\Phi_{|F|})$ and call the Mordell-Weil group of $\Phi_{|F|}$ whenever we regard 
${\rm MW}\,(J(\Phi_{|E|})) \subset {\rm Aut}\,(S)$ in the way explained here.

\end{remark}

The next corollary will be frequently used in Sections \ref{sect4} and \ref{sect5}. 

\begin{corollary}\label{cor32} 
Let $S$ be a projective K3 surface defined over an algebraically closed field $k$ of characteristic $p \not= 2$, $3$. Assume that there exist a smooth rational curve $C$, smooth rational curves $R_1$, $R_2$, possibly $R_1 = R_2$, and effective divisors $D_1$ and $D_2$, possibly $0$, such that 
\begin{enumerate}
\item The complete linear system $|E_i|$, where 
$$E_i := D_i + a_iR_i + b_iC$$ 
with suitable positive integers $a_i > 0$ and $b_i >0$ is free and defines a genus one fibration 
$$\Phi_i := \Phi_{|E_i|} : S \to \BP^1$$ 
of positive Mordell-Weil rank for $i = 1$ and $2$. 
\item Moreover, $\Phi_i$ ($i=1$, $2$) are different genus one fibrations, that is, the two classes $E_i$ ($i=1$, $2$) are not proportional in ${\rm Pic}\, (X) \simeq {\rm NS}\, (X)$. 
\item $C \not= R_1$, $C \not= R_2$ and $C \cap R_1 \cap R_2 \not= \emptyset$. 
\end{enumerate}
Then ${\rm Ine}\,(C)$ contains a non-commutative free subgroup isomorphic to $\Z * \Z$ and an element of positive entropy.
\end{corollary}

\begin{proof} Let us choose a point $P \in C \cap R_1 \cap R_2$, which exists by the assumption (3). By the assumption (1), we can choose an element $f_i \in {\rm MW}\, (\Phi_i)$ ($i=1$, $2$) of infinite order (and by replacing it by a suitable power if necessary) such that $f_i(C) = C$ and $f_i(P) = P$. Then $G := \langle f_1, f_2 \rangle$ is a subgroup of ${\rm Dec}\, (S, C, P)$. 

We now claim that $G$ does not preserves any genus one fibration on $S$. Our proof below is a slight modification of \cite[Theorem 1.2]{Og06} in which the existence of rational sections are assumed. . 

Indeed, otherwise, there is a genus one fibration $\Phi_{|E|} : S \to \BP^1$ preserving by both $f_1$ and $f_2$. Here $E$ is a genus one curve on $S$. By renumbering $i=1$ and $2$ if necessary, we may assume without loss of generality that $\Phi_{|E|}$ is a different genus one fibration from $\Phi_2$. Here we used the assumption (2). Then the classes $E$ and $E_2$ are nef. Since $E$ and $E_2$ are not proportional, it follows that
$$B := E + E_2$$
is a nef and big class in ${\rm Pic}\, (S) \simeq {\rm NS}\,(S)$. By the definition of $B$, the class $B$ is preserved by $f_2$. However, then, $f_2$ would be of finite order (see eg. \cite[Lemma 2.8]{Og07}). This contradicts to the fact that $f_2$ is of infinite order. Hence $G$ does not preserve any genus one fibration on $S$.

Note that $|G| = \infty$. Then $G$ is not almost abelian by Theorem \ref{thm30} (3). Since any subgroup of an almost abelian group is again almost abelian, the group ${\rm Dec}\,(S, C, P)$ is not almost abelian, either. The result now follows from Corollary \ref{cor31}.  
\end{proof}

\section{Singular K3 case}\label{sect4}

In this section, we work over the complex number field $\C$. Our main result of this section is Theorem \ref{thm:singIne}, which is the same as Theorem \ref{thm03} in Introduction. In this section and the next section, we use Kodaira's notation of singular fibers of genus one fibrations (See \cite[Page 565]{Ko63} for the notation). In this section and the next section, by our definition (Definition \ref{def}), an elliptic fibration always means a genus one fibration with a global section. 

\begin{theorem}\label{thm:singIne}
Let $X$ be a singular K3 surface. Then $X$ contains a smooth rational curve $C$ such that ${\rm Ine}\,(C)$ contains a non-commutative free subgroup isomorphic to $\Z * \Z$ and an element of positive entropy.
\end{theorem}

\begin{proof}

Let $d$ be the discriminant of $X$, i.e., the absolute value of the determinant of the N\'eron-Severi lattice ${\rm NS}\, (X)$, which is also the determinant of the transcendental lattice $T(X)$ of $X$. Then $d \ge 3$ and $X$ with the smallest two cases $d=3$ and $d=4$ are explicitly described in \cite{SI77}. In the case $d=3$, Theorem \ref{thm:singIne} is proved in \cite[Theorem 1.3]{Og18}. In the case $d=4$, the surface $X$ is a 2-elementary K3 surfaces with $\rho(X) = 20$ and $a(X) =2$ and Theorem \ref{thm:singIne} in this case will be proved in Section \ref{sect5}. 

In what follows, we assume $d\neq 3,4$ and we will closely follow the construction in \cite[Page184]{Og07}. 

First we find two elliptic fibrations whose Mordell-Weil groups are infinite.  In fact, the N\'eron-Severi lattice ${\rm NS}(X)$ of $X$ of the form:
$${\rm NS}\,(X)=U\oplus E_8^{\oplus 2}\oplus N\, ,$$ where $N$ is a negative definite lattice of rank 2. The lattice $U$ is an even unimodular hyperbolic lattice of rank $2$ and $E_8$ is the even unimodular negatice definite lattice of rank $8$. Note then that $d={\rm det}(N)$. Using this description of ${\rm NS}\,(X)$, one finds an elliptic fibration $\varphi: X\longrightarrow \PP^1$ whose reducible singular fibers are either 
$$II^* + II^*\,\, ,\,\, II^* +II^*+ I_2\,\, ,\,\, II^* + II^* + III\,\, .$$ See e.g. \cite[Lemma 2.1]{Kon92}.
Then rank of Mordell-Weil group of $\varphi$ is positive by the Shioda-Tate formula \cite[Corollary 5.3]{Shi90}. In particular, $\varphi$ admits at least two sections, say $D_1$ and $D_2$. Join two $II^*$ singular fibers of $\varphi_1$ by the section $D_1$ (resp. $D_2$) and throw out the components of multiplicity 2 at the edge of two $II^*$. Then one obtains a nef divisor of Kodaira's type $I_{12}^*$, say $E_1$ (resp. $E_2$), on $X$. The pencil $|E_1|$ (resp. $|E_2|$ ) gives rise to an elliptic fibration $\varphi_1:X\longrightarrow \PP^1$ (resp. $\varphi_2$). The two smooth rational curves $S_1$ and $S_2$ throwed out are actually the sections of $\varphi_i$ ($i=1$, $2$). We regard $S_1$ as the zero of ${\rm MW}\, (\varphi_i)$ and 
denote $P_i \in {\rm MW}\, (\varphi_i)$ 
the element corresponding to $S_2$. Then 
$${\rm rk\; MW}\, (\varphi_i) >0$$ 
for $i=1$ and $2$ by \cite[page 184]{Og07}. In fact, we checked there that 
$$\langle P_i, P_i \rangle > 0$$
for the height pairing on the Mordell-Weil group 
${\rm MW}\, (\varphi_i)$ (See \cite[Theorem 8.4, Definition 8.5]{Shi90} for the definition and \cite[Theorem 8.6, Table (8.16)]{Shi90} for the explicit formula we used). Thus $P_i \in {\rm MW}\, (\varphi_i)$ is of infinite order for $i=1$, $2$ by \cite[Page 228, (8.10)]{Shi90}. 

Let us choose two irreducible components $C$ and $C'$ of 
$${\rm Supp}\, E_1 \cap {\rm Supp}\, E_2\,\, $$ 
such that $C \not= C'$ and $C \cap C' \not= \emptyset$. Such $C$ and $C'$ certainly exist.
Now one can apply Corollary \ref{cor32} for $C$, $R_1=R_2 =C'$ and $E_i$ ($i=1$, $2$) to conclude 
the desired result.\end{proof}

\section{2-elementary K3 case}\label{sect5}

In this section, we assume that the base field is the complex number field $\C$. The main results of this section are Theorem \ref{thm:ge12} and Theorem \ref{thm:rho11} which are Theorem \ref{thm02} (1), (2).

Let $X$ be a 2-elementary K3 surface, i.e., $X$ is a complex projective K3 surface such that ${\rm NS}(X)^*/{\rm NS}(X)\cong (\Z/2\Z)^{\oplus a}$, for some positive integer $a=a(X)$, with Picard number $\rho=\rho(X)$. 2-elementary K3 surfaces are extensively studied by Nikulin. See for \cite[Section 4]{Ni81} about basic facts on 2-elementary K3 surfaces we will use. Since ${\rm NS}(X)$ is primitively embedded into the unimodular lattice $H^2(X,\Z)$, it follows that $\rho +a\le 22$. The K3 surface $X$ has an automorphism $\theta$ of order 2 such that  $$\theta^{*}|_{{\rm NS}\, (X)} = id_{{\rm NS}\, (X)}\,\, ,\,\, \theta^{*}\omega_X = -\omega_X\,\, .$$
Here $\omega_X$ is a nowhere vanishing holomorphic 2-form on $X$. The involution $\theta $ is in the center of ${\rm Aut}(X)$ and the fixed locus of $\theta$ is preserved under ${\rm Aut}(X)$. Following \cite[Definition 4.2.1]{Ni81}, the 2-elementary lattice ${\rm NS}(X)$ has an invariant $\delta$, where $\delta=0$ or $1$.

\begin{lemma}(\cite[Remark 4.4]{Og18})\label{lem:22}
If $X$ contains a smooth rational curve and ${\rm Aut}(X)$ has an element of positive entropy, then $\rho+a=22$.
\end{lemma}

\begin{proof}
Suppose  $X$ contains a smooth rational curve and ${\rm Aut}(X)$ has an element of positive entropy. Then $(\rho,a,\delta)\neq (10,10,0)$ (otherwise, ${\rm NS}(X)\cong U(2)\oplus E_8(2)$ and $X$ contains no smooth rational curve, a contradiction) and $(\rho,a,\delta)\neq (10,8,0)$ (otherwise, by \cite[Theorem 4.2.2]{Ni81}, $X^\theta$ is the disjoint union of two smooth elliptic curves, then, by \cite[Theorem 1.4]{Og07}, ${\rm Aut}(X)$ has no element of positive entropy, a contradiction). Then, by \cite[Theorem 4.2.2]{Ni81} again, $$X^\theta=H+\sum_{i=1}^{k}C_i$$ where $H$ is a smooth projective curve of genus $g_H=(22-\rho-a)/2$, $C_i$ are smooth rational curves, and $k=(\rho-a)/2$.  If $\rho+a<22$, then $g_H>0$ and, by \cite[Theorem 1.4]{Og07},  ${\rm Aut}(X)$ no element of positive entropy, a contradiction. Thus, $\rho+a=22$. This completes the proof of the lemma.
\end{proof}

 We are interested in an inerita group with an element of positive entropy (see Theorem \ref{thm21}). {\it Thus, in the rest of this section, we assume that $$\rho + a =22.$$} Then the locus of fixed points of $\theta$ $$X^{\theta}=\cup_{i=1}^{i=k}C_i$$ where, $k=(\rho-a+2)/2$, and $C_i$ are disjoint smooth rational curves. Let $$C=C_1+...+C_k.$$  We use $A_l\, (l\ge 1),D_m \,(m\ge 4), E_{n}\, (n=6,7,8)$ to denote a negative definite root lattice whose basis is given by the corresponding Dynkin diagram. By classification of 2-elementary lattices (\cite[Theorem 4.3.2]{Ni81}), there are exactly 11 cases for the triple $(\rho, a,\delta)$ with $\rho +a=22$ (\cite[Section 4, Table 1]{Ni81}):

\begin{enumerate}
\item $(\rho, a,\delta)=(11,11,1)$ (then $k=1$, ${\rm NS}(X)\cong U(2)\oplus A_1^{\oplus 9}$);

\item $(\rho, a,\delta)=(12,10,1)$ (then $k=2$, ${\rm NS}(X)\cong U\oplus A_1^{\oplus 10}$);

\item $(\rho, a,\delta)=(13,9,1)$ (then $k=3$,  ${\rm NS}(X)\cong U\oplus D_4\oplus A_1^{\oplus 7}$);

\item $(\rho, a,\delta)=(14,8,1)$ (then $k=4$,  ${\rm NS}(X)\cong U\oplus D_4\oplus D_4\oplus A_1^{\oplus 4}$);

\item $(\rho, a,\delta)=(15,7,1)$ (then $k=5$,  ${\rm NS}(X)\cong U\oplus D_4^{\oplus 3}\oplus A_1$);

\item $(\rho, a,\delta)=(16,6,1)$ (then $k=6$,  ${\rm NS}(X)\cong U\oplus D_6^{\oplus 2}\oplus A_1^{\oplus 2}$);

\item $(\rho, a,\delta)=(17,5,1)$ (then $k=7$,  ${\rm NS}(X)\cong U\oplus D_6\oplus D_8\oplus A_1$);

\item $(\rho, a,\delta)=(18,4,0)$ (then $k=8$,  ${\rm NS}(X)\cong U\oplus D_4\oplus D_{12}$);

\item $(\rho, a,\delta)=(18,4,1)$ (then $k=8$,  ${\rm NS}(X)\cong U\oplus D_{14}\oplus A_1^{\oplus 2}$);

\item  $(\rho, a,\delta)=(19,3,1)$ (then $k=9$,  ${\rm NS}(X)\cong U\oplus D_{16}\oplus A_1$);

\item $(\rho, a,\delta)=(20,2,1)$ (then $k=10$, ${\rm NS}(X)\cong U\oplus E_8\oplus D_{10}$).
\end{enumerate}

\medskip

The following lemma is due to \cite[Section 4]{Ni81}.

\begin{lemma}\label{lem:2points}
Let  $H$ be a smooth rational curve on $X$ different from $C_i$ for all $1\le i\le k $. Then $\theta(H)=H$, and $C$ and $H$ meet at exactly two points transversally. In particular, $C.H=2$.
\end{lemma}

\begin{proof}
Following \cite[Section 4]{Ni81}, we recall a proof. Since $\theta$ acts trivially on ${\rm NS}(X)$ and $H$ is the unique element of the complete linear system $|H|$, it follows that $\theta (H)=H$. Note that any involution of $\PP^1$ has exactly two fixed points. Since $X^{\theta}=C$, it follows that $C$ and $H$ intersect at exactly two points. At each point, say $P$, of the two intersection points, the tangent direction of $C$ (resp. $H$) corresponds to the eigenvector of the induced action $\theta^*|T_{X,P}$ with respect to eigenvalue $1$ (resp. $-1$). Thus, $C$ and $H$ meet transversally  at $P$.\end{proof}

\begin{lemma}\label{lem:Jacobian}
Let $\varphi:X\longrightarrow \PP^1$ be an elliptic fibration with a section $H$. Suppose there exists $i\in\{1,2,..,k\}$ such that $C_i$ is not contained in any fiber of $\varphi$. Then $H\subset C$. Moreover, $C.F=4$ for any fiber $F$ of $\varphi$.
\end{lemma}

\begin{proof}
Since $C_i$ is not contained  in any fiber of $\varphi$, it follows that $C_i$ intersects with each fiber of $\varphi$. Since $C_i\subset X^\theta$, it follows that $\theta$ preserves each fiber of $\varphi$, i.e., $\varphi \circ \theta=\varphi$.  Note also that $\theta(H)=H$ by Lemma \ref{lem:2points}. Then $\theta |_H=id_H$ and $H\subset C$.

Note that for a general fiber, say $E$, of $\varphi$, $\theta|_E$ is just the inversion of the elliptic curve $E$ (view $E\cap H$ as zero $O$ of $E$). Thus $\theta|_E$ has exactly 4 fixed points. Then $E$ meets with $C$ at four points transversally. Thus $C.F=C.E=4$ for any fiber $F$ of $\varphi$. \end{proof}

In the next lemma, we need the notion of the Mordell-Weil group ${\rm MW}(\varphi)$ of a genus one fibration $\varphi: X\longrightarrow \PP^1$ explained in Remark \ref{rmk:35}.

\begin{lemma}\label{lem:positiveMW}
Let $E$ be an effective reducible divisor such that the complete linear system $|E|$ defines a genus one fibration $\varphi: X\longrightarrow \PP^1$ . Let $\sI=\{i | C_i \subset {\rm Supp} E\}$. Let $r$ be the number of irreducible components of $E$. If either one of the following (1) or (2) holds, then ${\rm rk \; MW}(\varphi)>0$.
\begin{enumerate}

 \item  the cardinality $|\sI|$ is equal to $k$ and $r<\rho(X)-1$, or
  
 \item  $|\sI|=k-1$, $C_{i}.E=0$ for the unique $i\in \{1,...,k\}\setminus \sI$, and $r <\rho(X) -2$.\end{enumerate}
  
  \end{lemma}

\begin{proof}
Case (1): By Lemma \ref{lem:2points},  $E$ is the only reducible singular fiber of $\varphi$ (otherwise, suppose $H$ is a smooth rational curve in a reducible fiber different from $E$, then $H.E=0$ and $H.C=0$, a contradiction). Then by $r<\rho(X)-1$ and by \cite[Lemma 2.1]{Yu18},  $\varphi$ has infinite automorphism group, and ${\rm rk \; MW}(\varphi)>0$. 

Case (2):  Since $C_{i}.E=0$, it follows that $C_{i}$ is in a reducible fiber, say $F$, of $\varphi$. Let $H$ be a component of $F$ different from $C_{i}$. Then $H.E=0$. Since $|\sI|=k-1$, it follows that $2=H.C=H.C_{i}$. Thus, $F=C_{i}+H$. Similar to Case (1), only $E$ and $F$ are the reducible fibers of $\varphi$. Then by $r <\rho(X)-2$ and by \cite[Lemma 2.1]{Yu18}, ${\rm rk \; MW}(\varphi)>0$. \end{proof}

Theorem \ref{thm02} (1) follows from:

\begin{theorem}\label{thm:ge12}
If $\rho\ge 12$, then there exists a smooth rational curve $H\subset X$ such that ${\rm Ine}\,(H)$ contains a non-commutative free subgroup isomorphic to $\Z * \Z$ and an element of positive entropy.
\end{theorem}

\begin{proof}

Case {$(\rho,a,\delta)=(12,10,1)$}: Then $k=2$ and $C=C_1+C_2$. By  ${\rm NS}(X)\cong U\oplus A_1^{\oplus 10}$, there exists an elliptic fibration $\varphi:X\longrightarrow \PP^1$ with exactly 10 reducible fibers: $H_i+H_i^\prime$, $1\le i\le 10$, where $H_i$ and $H_i^\prime$ are smooth rational curves. By Lemma \ref{lem:2points}, at least one of $C_1$ and $C_2$ is not contained in any fiber of $\varphi$. Then, by Lemma  \ref{lem:Jacobian}, one of the two components of $C$ is a section of $\varphi$. Interchanging $C_1$ and $C_2$ if necessary, we may assume $C_1$ is a section of $\varphi$. Then, for any $i$, interchanging $H_i$ and $H_i^\prime$ if necessary, we may assume $C_1.H_i=1$. By Lemma \ref{lem:2points}, $C.H_i=2$. Thus $C_2.H_i=1$.  Let $$E_1=C_1+H_1+C_2+H_2$$ and $$E_2=C_1+H_1+C_2+H_3,$$ which are of Kodaira type $I_4$. By Lemma \ref{lem:positiveMW} (1), $|E_1|$ and $|E_2|$ define two genus one fibrations of positive Mordell-Weil rank. Thus, by applying Corollary \ref{cor32} for $E_1,E_2,C=C_1,R_1=R_2=H_1$, we deduce that ${\rm Ine}\,(C_1)$ contains a non-commutative free subgroup isomorphic to $\Z * \Z$ and an element of positive entropy. Thus, this case is proved.

Note that the choice of $H$ in Theorem \ref{thm:ge12} is not necessarily unique. In fact, for the same reason, each of ${\rm Ine} (C_2)$ and ${\rm Ine} (H_1)$ contains a non-commutative free subgroup isomorphic to $\Z * \Z$ and an element of positive entropy.


Case {$(\rho,a,\delta)=(13,9,1)$}: Then $k=3$ and $C=C_1+C_2+C_3$. By  ${\rm NS}(X)\cong U\oplus D_4\oplus A_1^{\oplus 7}$, there exists an elliptic fibration $\varphi:X\longrightarrow \PP^1$ with exactly 8 reducible fibers: $$2F_0+F_1+F_2+F_3+F_4\;({\rm Kodaira \; type}\; I_0^*), H_1+H_1^\prime,\ldots, \; H_7+H_7^\prime.$$ As in the previous case, by Lemmas \ref{lem:2points}, \ref{lem:Jacobian}, we may assume $C_1$ is a section of $\varphi$. Then we may assume $C_1.F_1=1$ and $C_1.H_i=1$ for all $i$. Since $\theta$ preserves any smooth rational curve on $X$ (Lemma \ref{lem:2points}), it follows that each of the four points $F_0\cap F_i$, $1\le i\le 4$, must be fixed by $\theta$. Then $F_0$ must be fixed by $\theta$ pointwisely. Thus $F_0\subset C$, and we may assume $F_0=C_2$. Thus $C_2.H_i=0$. Then by Lemma \ref{lem:2points}, $C_3.H_i=1$ for all $1\le i\le 7$. Note that since $C_1$ is a section of $\varphi$ and $C_1.F_1=1$, it follows that $C_1.F_i=0$, for all $i=2,3,4$. Then $C_3.F_i=1$, for all $i=2,3,4$. Note that $C_3$ is a 3-section of $\varphi$. Let $$E_1:=C_2+F_1+C_1+H_1+C_3+F_2 \;({\rm Kodaira \; type}\; I_6)$$ and $$E_2:=C_2+F_1+C_1+H_2+C_3+F_2 \; ({\rm Kodaira \; type}\; I_6).$$  By Lemma \ref{lem:positiveMW} (1), $|E_1|$ and $|E_2|$ define two different genus one fibrations of positive Mordell-Weil rank. Thus, by applying Corollary \ref{cor32} for $E_1,E_2,C=C_1,R_1=R_2=F_1$, we deduce that ${\rm Ine}\,(C_1)$ contains a non-commutative free subgroup isomorphic to $\Z * \Z$ and an element of positive entropy. Thus, this case is proved.

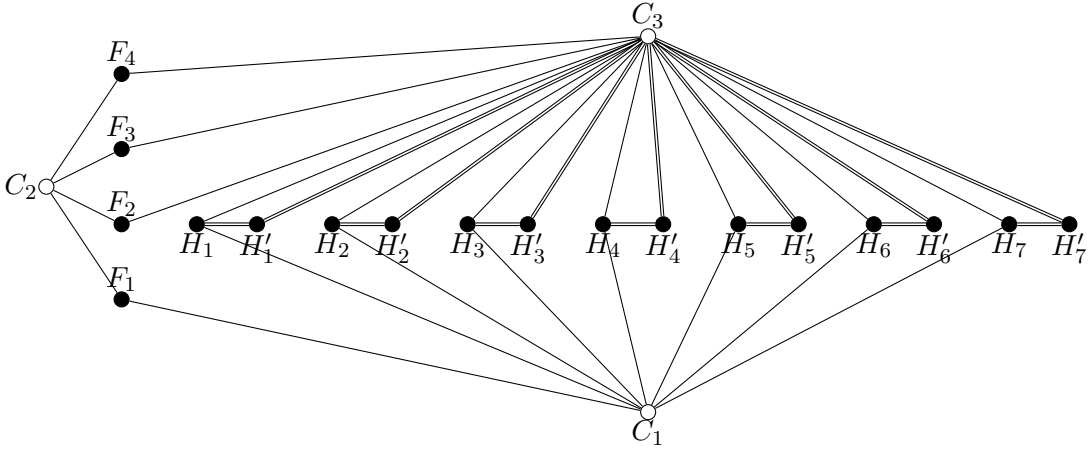
\begin{figure}

\begin{tikzpicture}[inner sep=0pt,minimum size=2mm]

\path  node(C2) at ( 0,-0.5) [circle,,draw,label=left:$C_2$] {}
  node(F4) at ( 1,1) [circle,fill=black,draw,label=above:$F_4$] {}
  node(F3) at ( 1,0) [circle,fill=black,draw,label=above:$F_3$] {}
  node(F2) at ( 1,-1) [circle,fill=black,draw,label=above:$F_2$] {}
  node(F1) at (1,-2) [circle,fill=black,draw,label=above:$F_{1}$] {}
  node(C3) at ( 8,1.5) [circle,draw,label=above:$C_3$] {}
  node(C1) at (8,-3.5) [circle,draw,label=below:$C_1$] {}
  node(H1) at ( 2,-1) [circle,fill=black,draw,label=below:$H_1$] {}
  node(H1p) at ( 2.8,-1) [circle,fill=black,draw,label=below:$H_1^\prime$] {}
   node(H2) at ( 3.8,-1) [circle,fill=black,draw,label=below:$H_2$] {}
  node(H2p) at (4.6,-1) [circle,fill=black,draw,label=below:$H_2^\prime$] {}
   node(H3) at (5.6,-1) [circle,fill=black,draw,label=below:$H_3$] {}
  node(H3p) at (6.4,-1) [circle,fill=black,draw,label=below:$H_3^\prime$] {}
   node(H4) at (7.4,-1) [circle,fill=black,draw,label=below:$H_4$] {}
  node(H4p) at (8.2,-1) [circle,fill=black,draw,label=below:$H_4^\prime$] {}
   node(H5) at (9.2,-1) [circle,fill=black,draw,label=below:$H_5$] {}
  node(H5p) at (10,-1) [circle,fill=black,draw,label=below:$H_5^\prime$] {}
   node(H6) at (11,-1) [circle,fill=black,draw,label=below:$H_6$] {}
  node(H6p) at (11.8,-1) [circle,fill=black,draw,label=below:$H_6^\prime$] {}
   node(H7) at (12.8,-1) [circle,fill=black,draw,label=below:$H_7$] {}
  node(H7p) at (13.6,-1) [circle,fill=black,draw,label=below:$H_7^\prime$] {};

\draw[double] (node cs:name=H1) -- (node cs:name=H1p);
\draw[double] (node cs:name=H2) -- (node cs:name=H2p);
\draw[double] (node cs:name=H3) -- (node cs:name=H3p);
\draw[double] (node cs:name=H4) -- (node cs:name=H4p);
\draw[double] (node cs:name=H5) -- (node cs:name=H5p);
\draw[double] (node cs:name=H6) -- (node cs:name=H6p);
\draw[double] (node cs:name=H7) -- (node cs:name=H7p);

\draw[double] (node cs:name=H1p) -- (node cs:name=C3);
\draw[double] (node cs:name=H2p) -- (node cs:name=C3);
\draw[double] (node cs:name=H3p) -- (node cs:name=C3);
\draw[double] (node cs:name=H4p) -- (node cs:name=C3);
\draw[double] (node cs:name=H5p) -- (node cs:name=C3);
\draw[double] (node cs:name=H6p) -- (node cs:name=C3);
\draw[double] (node cs:name=H7p) -- (node cs:name=C3);

\draw (node cs:name=H1) -- (node cs:name=C3);
\draw (node cs:name=H2) -- (node cs:name=C3);
\draw (node cs:name=H3) -- (node cs:name=C3);
\draw (node cs:name=H4) -- (node cs:name=C3);
\draw (node cs:name=H5) -- (node cs:name=C3);
\draw (node cs:name=H6) -- (node cs:name=C3);
\draw (node cs:name=H7) -- (node cs:name=C3);

\draw (node cs:name=H1) -- (node cs:name=C1);
\draw (node cs:name=H2) -- (node cs:name=C1);
\draw (node cs:name=H3) -- (node cs:name=C1);
\draw (node cs:name=H4) -- (node cs:name=C1);
\draw (node cs:name=H5) -- (node cs:name=C1);
\draw (node cs:name=H6) -- (node cs:name=C1);
\draw (node cs:name=H7) -- (node cs:name=C1);

\draw (node cs:name=C2) -- (node cs:name=F1);
\draw (node cs:name=C2) -- (node cs:name=F2);
\draw (node cs:name=C2) -- (node cs:name=F3);
\draw (node cs:name=C2) -- (node cs:name=F4);

\draw (node cs:name=C1) -- (node cs:name=F1);
\draw (node cs:name=C3) -- (node cs:name=F2);
\draw (node cs:name=C3) -- (node cs:name=F3);
\draw (node cs:name=C3) -- (node cs:name=F4);


\end{tikzpicture}
\caption{$C_1,C_3$, and reducible singular fibers of $\varphi$ for case $\rho=13$.}
\end{figure}

 All of the remaining cases (i.e., cases $\rho\ge 14$) can be proved similarly as in the case {$(\rho,a,\delta)=(13,9,1)$}. For each of the remaining cases, we just list the reducible singlar fibers of an elliptic fibration $\varphi$ on $X$, from which we start as in the previous case.  It turns out that, in all of these remaining cases, a component, say $C_1$, of $C$ is a section of $\varphi$, and another component, say $C_k$, of $C$ is a 3-section of $\varphi$.  Then we give the definition of two effective divisors $E_1$ and $E_2$ for which we apply Corollary \ref{cor32}. Then by Corollary \ref{cor32}, we see that ${\rm Ine}\,(C_k)$ contains a non-commutative free subgroup isomorphic to $\Z * \Z$ and an element of positive entropy.
 
 {\it In the rest of the proof, for $1\le i,j\le k$, we use $G_{ij}$, $F_{ij}$, $F_{ij}^\prime$, etc. to denote smooth rational curves which intersect both $C_i$ and $C_j$ ($i=j$ means intersecting with $C_i$ at two distinct points). }

Case {$(\rho,a,\delta)=(14,8,1)$}: Then $k=4$ and $C=C_1+C_2+C_3+C_4$. By  ${\rm NS}(X)\cong U\oplus D_4^{\oplus 2}\oplus A_1^{\oplus 4}$, there exists an elliptic fibration (with a section $C_1$) $\varphi:X\longrightarrow \PP^1$ with exactly 6 reducible fibers: $2C_2+F_{12}+F_{24}+F_{24}^\prime+F_{24}^{\prime\prime}$ (type $I_0^*$), $2C_3+F_{13}+F_{34}+F_{34}^\prime+F_{34}^{\prime\prime}$ (type $I_0^*$), $F_{14}+F_{44}$, $F_{14}^\prime+F_{44}^\prime$, $F_{14}^{\prime\prime}+F_{44}^{\prime\prime}$, $F_{14}^{\prime\prime\prime}+F_{44}^{\prime\prime\prime}$. Let $E_1:=C_2+F_{12}+C_1+F_{14}+C_4+F_{24}$ (type $I_6$) and  $E_2:=C_2+F_{12}+C_1+F_{14}^\prime+C_4+F_{24}$ (type $I_6$). Note that in this case, we use Lemma \ref{lem:positiveMW} (2) to prove the positivity of Mordell-Weil ranks.

Case {$(\rho,a,\delta)=(15,7,1)$}: Then $k=5$ and $C=C_1+...+C_5$. By  ${\rm NS}(X)\cong U\oplus D_4^{\oplus 3}\oplus A_1$,  there exists an elliptic fibration (with a section $C_1$) $\varphi:X\longrightarrow \PP^1$ with exactly 4 reducible fibers: $2C_2+F_{12}+F_{25}+F_{25}^\prime+F_{25}^{\prime\prime}$ (type $I_0^*$), $2C_3+F_{13}+F_{35}+F_{35}^\prime+F_{35}^{\prime\prime}$ (type $I_0^*$), $2C_4+F_{14}+F_{45}+F_{45}^\prime+F_{45}^{\prime\prime}$ (type $I_0^*$), $F_{15}+F_{55}$.  Let $E_1:=C_4+F_{45}+C_5+F_{25}+C_2+F_{12}+C_1+F_{14}$ (type $I_8$) and $E_2:=C_4+F_{45}+C_5+F_{25}^\prime+C_2+F_{12}+C_1+F_{14}$ (type $I_8$). Note that in this case, we use Lemma \ref{lem:positiveMW} (2) to prove the positivity of Mordell-Weil ranks.

Case {$(\rho,a,\delta)=(16,6,1)$}: Then $k=6$ and $C=C_1+...+C_6$. By  ${\rm NS}(X)\cong U\oplus D_6^{\oplus 2}\oplus A_1^{\oplus 2}$, there exists an elliptic fibration (with a section $C_1$) $\varphi:X\longrightarrow \PP^1$ with exactly 4 reducible fibers: $2C_2+2G_{23}+2C_3+F_{13}+F_{36}+F_{26}+F_{26}^\prime$ (type $I_2^*$),  $2C_4+2G_{45}+2C_5+F_{15}+F_{56}+F_{46}+F_{46}^\prime$ (type $I_2^*$), $F_{16}+F_{66}$, $F_{16}^\prime+F_{66}^\prime$.  Let $E_1:=C_3+F_{13}+C_1+F_{15}+C_5+G_{45}+C_4+F_{46}+C_6+F_{26}+C_2+G_{23}$ (type $I_{12}$) and $E_2:=C_3+F_{13}+C_1+F_{15}+C_5+G_{45}+C_4+F_{46}+C_6+F_{26}^{\prime}+C_2+G_{23}$ (type $I_{12}$).

Case {$(\rho,a,\delta)=(17,5,1)$}: Then $k=7$ and $C=C_1+C_2+...+C_7$. By  ${\rm NS}(X)\cong U\oplus D_6\oplus D_8\oplus A_1$, there exists an elliptic fibration (with a section $C_1$) $\varphi:X\longrightarrow \PP^1$ with exactly 3 reducible fibers: $F_{27}+F_{27}^\prime+2C_2+2G_{23}+2C_3+F_{37}+F_{31}$ (type $I_2^*$), $F_{47}+F_{47}^\prime+2C_4+2G_{45}+2C_5+2G_{56}+2C_6+F_{61}+F_{67}$ (type $I_4^*$), $F_{17}+F_{77}$.  Let $E_1:=C_4+G_{45}+C_5+G_{56}+C_6+F_{61}+C_1+F_{31}+C_3+G_{23}+C_2+F_{27}+C_7+F_{47}$ (type $I_{14}$) and $E_2:=C_4+G_{45}+C_5+G_{56}+C_6+F_{61}+C_1+F_{31}+C_3+G_{23}+C_2+F_{27}+C_7+F_{47}^\prime$ (type $I_{14}$).

Case {$(\rho,a,\delta)=(18,4,0)$}: This is the case studied in \cite{Og18}. Then $k=8$ and $C=C_1+C_2+...+C_8$. By  ${\rm NS}(X)\cong U\oplus D_4\oplus D_{12}$,  there exists an elliptic fibration (with a section $C_1$) $\varphi:X\longrightarrow \PP^1$ with exactly 2 reducible fibers: $F_{28}+F_{28}^\prime+2C_2+F_{28}^{\prime\prime}+F_{12}$ (type $I_0^*$), $F_{38}+F_{38}^\prime+2C_3+2G_{34}+2C_4+2G_{45}+2C_5+2G_{56}+2C_6+2G_{67}+2C_7+F_{17}+F_{78}$ (type $I_8^*$).  Let $E_1:=C_3+G_{34}+C_4+G_{45}+C_5+G_{56}+C_6+G_{67}+C_7+F_{17}+C_1+F_{12}+C_2+F_{28}^{\prime\prime}+C_8+F_{38}$ (type $I_{16}$) and $E_2:=C_3+G_{34}+C_4+G_{45}+C_5+G_{56}+C_6+G_{67}+C_7+F_{17}+C_1+F_{12}+C_2+F_{28}^{\prime\prime}+C_8+F_{38}^\prime$ (type $I_{16}$).

Case {$(\rho,a,\delta)=(18,4,1)$}: Then $k=8$ and $C=C_1+C_2+...+C_8$. By  ${\rm NS}(X)\cong U\oplus D_{14}\oplus A_1^{\oplus 2}$,  there exists an elliptic fibration (with a section $C_1$) $\varphi:X\longrightarrow \PP^1$ with exactly 3 reducible fibers:  $F_{28}+F_{28}^\prime+2C_2+2G_{23}+2C_3+2G_{34}+2C_4+2G_{45}+2C_5+2G_{56}+2C_6+2G_{67}+2C_7+F_{78}+F_{17}$ (type $I_{10}^*$), $F_{18}+F_{88}$, $F_{18}^\prime+F_{88}^\prime$.  Let $E_1:=C_2+G_{23}+C_3+G_{34}+C_4+G_{45}+C_5+G_{56}+C_6+G_{67}+C_7+F_{17}+C_1+F_{18}+C_8+F_{28}$ (type $I_{16}$) and $E_2:=C_2+G_{23}+C_3+G_{34}+C_4+G_{45}+C_5+G_{56}+C_6+G_{67}+C_7+F_{17}+C_1+F_{18}+C_8+F_{28}^\prime$ (type $I_{16}$).

Case {$(\rho,a,\delta)=(19,3,1)$}: Then $k=9$ and $C=C_1+C_2+...+C_9$. By  ${\rm NS}(X)\cong U\oplus D_{16}\oplus A_1$,  there exists an elliptic fibration (with a section $C_1$) $\varphi:X\longrightarrow \PP^1$ with exactly 2 reducible fibers:  $F_{29}+F_{29}^\prime+2C_2+2G_{23}+2C_3+2G_{34}+2C_4+2G_{45}+2C_5+2G_{56}+2C_6+2G_{67}+2C_7+2G_{78}+2C_8+F_{89}+F_{18}$ (type $I_{12}^*$), $F_{19}+F_{99}$. Let $E_1:=C_2+G_{23}+C_3+G_{34}+C_4+G_{45}+C_5+G_{56}+C_6+G_{67}+C_7+G_{78}+C_8+F_{89}+C_9+F_{29}$ (type $I_{16}$) and $E_2:=C_2+G_{23}+C_3+G_{34}+C_4+G_{45}+C_5+G_{56}+C_6+G_{67}+C_7+G_{78}+C_8+F_{89}+C_9+F_{29}^\prime$ (type $I_{16}$).  Note that in this case, we use Lemma \ref{lem:positiveMW} (2) to prove the positivity of Mordell-Weil ranks.

Case {$(\rho,a,\delta)=(20,2,1)$}:  This is the case where $X$ is a singular K3 surface of discriminant 4 (cf. Section \ref{sect4}). Then $k=10$ and $C=C_1+...+C_9+C_{0}$, where $C_{0}:=C_{10}$.  By  ${\rm NS}(X)\cong U\oplus E_8\oplus D_{10}$, there exists an elliptic fibration (with a section $C_1$) $\varphi:X\longrightarrow \PP^1$ with exactly 2 reducible fibers:  $2C_2+4G_{23}+6C_3+3F_{30}+5G_{34}+4C_4+3G_{45}+2C_5+F_{15}$ (type $II^*$), $F_{16}+F_{60}+2C_6+2G_{67}+2C_7+2G_{78}+2C_8+2G_{89}+2C_9+F_{90}+F_{90}^\prime$ (type $I_{6}^*$). Let $E_1:=C_3+2F_{30}+C_6+2F_{60}+C_9+2F_{90}+3C_0$ (type $IV^*$) and $E_2:=C_3+2F_{30}+C_6+2F_{60}+C_9+2F_{90}^\prime+3C_0$ (type $IV^*$). Then $|E_1|$ and $|E_2|$ define two different elliptic fibrations $\varphi_i : X\longrightarrow \PP^1$, $i=1,2$. Note that, for both $i=1,2$, $G_{23}$ and $G_{34}$ are two disjoint sections of $\varphi_i$ meeting the same irreducible component $C_3$ of the fiber $E_i$ of $\varphi_i$. Since the group structure of type $IV^*$ fiber is $\C \times \Z/3$ by \cite[Table 1, p. 604]{Ko63}, regarding $O = [G_{23}] \in {\rm MW}\, (\varphi_i)$, the element $[G_{34}] \in {\rm MW}\,(\varphi_i)$ is of infinite order for $i = 1$, $2$. Thus, both $\varphi_1$ and $\varphi_2$ are of positive Mordell-Weil rank. One can now apply Corollary \ref{cor32} to conclude. 

This completes the proof of the theorem.\end{proof}

In the rest of this section, we consider the case $(\rho, a, \delta)=(11,11,1)$. Then $k=1$ and $C=X^\theta$ is an irreducible smooth rational curve. Note that ${\rm Aut}(X)={\rm Dec}(C)$, because $\theta$ is in the center of ${\rm Aut}(X)$.

\begin{lemma}\label{lem:triviallyonNS}
 Let $\phi: {\rm Aut(}X)\longrightarrow {\rm O}({\rm NS}(X))$ be the map given by $\phi (f)=f^*|_{{\rm NS}(X)}$. Then ${\rm Ker}\; \phi=\{id_X, \theta \}$.\end{lemma}

\begin{proof}
Since ${\rm rk}\; T(X)=11$ an odd number, it follows that, for any automorphism $g$ of $X$, the induced action $g^*$ on $T(X)$ must be $\pm id_{T(X)}$. Thus, ${\rm Ker}\; \phi=\{id_X, \theta \}$. \end{proof}

\begin{lemma}\label{lem:Qbasis}
There exist eleven smooth rational curves such that their classes in ${\rm NS}(X)$ form a $\Q$-basis of ${\rm NS}(X)_{\Q}$. 
\end{lemma}

\begin{proof}
By ${\rm NS}(X)\cong U(2)\oplus A_1^{\oplus 9}$, there exists a genus one fibration, say $\varphi: X\longrightarrow \PP^1$, such that this fibration has exactly 9 reducible singular fibres and each of them consists of exactly two irreducible components, say $H_i$ and $H_i^\prime$, $1\le i\le 9$,  where $H_i$ and $H_i^\prime$ are smooth rational curves.  By Lemma \ref{lem:2points}, $C$ cannot be contained in any fiber of $\varphi$. Then $$ [H_1]+[H_1^\prime], [C], [H_1],[H_2],...,[H_9]$$ form $\Q$-basis of ${\rm NS}(X)_{\Q}$ since the corrsponding Gram matrix has nonzero determinant. Thus $$[H_1^\prime], [C], [H_1],[H_2],...,[H_9]$$ also form a $\Q$-basis of ${\rm NS}(X)_{\Q}$. This completes the proof of the  lemma.\end{proof}

Theorem \ref{thm02} (2), hence Theorem \ref{thm01}, follows from:

\begin{theorem}\label{thm:rho11}
Let $\rho: {\rm Aut(}X)={\rm Dec}(C)\longrightarrow {\rm Aut}(C)$ be the restriction map. If ${\rm Ker}\; \rho \neq \{id_X, \theta \}$, i.e., ${\rm Ine}(C)\neq \{id_X, \theta \}$ , then ${\rm Ine}\,(C)$ contains a non-commutative free subgroup isomorphic to $\Z * \Z$ and an element of positive entropy.
\end{theorem}

\begin{proof}
Suppose ${\rm Ine}(C) \neq \{id_X, \theta \}$. Let $g\in {\rm Ine}(C) \setminus \{id_X, \theta\} $. By Lemmas \ref{lem:triviallyonNS} and \ref{lem:Qbasis}, there exists a smooth rational curve $H$ on $X$ such that $g(H)\neq H$. By $g|_C=id_C$ and Lemma \ref{lem:2points}, $H\cap g(H)\cap C$ consists of exactly two points, say $P_1$ and $P_2$. Let $E_1:=H+C$ and $E_2:=g(H)+C$. Then complete linear systems $|E_i|$, $i=1,2$, define two genus one fibrations, say $\varphi_i: X\longrightarrow \PP^1$, of positive Mordell-Weil rank (Lemma \ref{lem:positiveMW}). Then by Corollary \ref{cor32}, ${\rm Ine}\,(C)$ contains a non-commutative free subgroup isomorphic to $\Z * \Z$ and an element of positive entropy. This completes the proof of the theorem.\end{proof}

\end{document}